\documentclass[11pt, a4paper]{amsart}

\RequirePackage{amssymb,amsfonts, amsmath, amsthm,latexsym}
\RequirePackage[british]{babel}
\RequirePackage{version,graphicx,times,titletoc,dirtytalk}
\RequirePackage[utf8]{inputenc}
\RequirePackage{xcolor, url, hyperref}

\newtheorem {thm}{Theorem}
\newtheorem {cor}[thm]{Corollary}
\newtheorem {lem}[thm]{Lemma}
\newtheorem {prop}[thm]{Proposition}

\theoremstyle{definition}
\newtheorem {defi}[thm]{Definition}
\newtheorem {rem}[thm]{Remark}
\newtheorem {exa}[thm]{Example}

\DeclareMathOperator{\Gal}{Gal}

\DeclareMathOperator{\lcm}{lcm}

\newcommand{\Q}{\mathbb{Q}}

\newcommand{\ch}{\mathrm{char}}

\renewcommand{\geq}{\geqslant}
\renewcommand{\leq}{\leqslant}

\title{Linear relations among radicals}
\author{Antonella Perucca}
\date{}


\begin{document}

\maketitle

\begin{center}
     \textit{In memory of Marc Rybowicz}
\end{center}

\begin{abstract}
Let $K$ be a field, fix an algebraic closure $\overline{K}$, and let $G$ be a subgroup of $\overline{K}^\times$. We are able to give a closed formula for the ratio between the degree $[K(G):K]$ and the index $|GK^\times:K^\times|$, provided that the latter is finite. Our formula explains all the $K$-linear relations among radicals, which (beyond the ones stemming from the multiplicative group $GK^\times/K^\times$) are generated by relations among roots of unity and single radicals.
Our work builds on results by Rybowicz, which in turn are based on Kneser's theorem on the linear independence of radicals. 
\end{abstract}

\section{Introduction}

We let $K$ be a field, for which we fix an algebraic closure $\overline{K}$. We consider the \emph{radicals} over $K$, by which we mean the elements $\alpha\in \overline{K}^\times$ for which there exists a positive integer $n$ such that $\alpha^n\in K^\times$. 
We fix a group $G$ of radicals such that the \emph{index} $|GK^\times :K^\times|$ is finite and investigate the \emph{degree} $[K(G):K]$. (Note that the extension $K(G)/K$ is in general not Galois, but it is separable if the index is not divisible by the characteristic.) More precisely, we compare the index and the degree, keeping in mind that the former (which is usually  easier to compute) will let us understand the latter. What is true in general is that 
\begin{equation}\label{maininequality}
[K(G):K] \leq |GK^\times :K^\times|    
\end{equation}
because representants for $GK^\times/K^\times$ can be taken as generators for the $K$-vector space $K(G)$ (and indeed the multiplicative relations that determine the index also affect the degree). Having a strict inequality in \eqref{maininequality} is a phenomenon that is called \emph{entanglement} of radicals.
Note that the degree does not necessarily divide the index: for example, if $K=\mathbb Q$ and $p$ is an odd prime number and $G$ is generated by a root of unity of order $p$, then the ratio degree/index is $(p-1)/p$.

Consider the setting of Kummer theory, namely suppose that the exponent $n$ of the group $GK^\times/K^\times$ has the following two properties: $n$ is not divisible by the characteristic of $K$; $K^\times$ contains a root of unity of order $n$. In this case, the group $GK^\times/K^\times$ is isomorphic to $G^nK^{\times n}/K^{\times n}$ hence the core result of Kummer theory states that we have an equality in \ref{maininequality} (see for example \cite[Chapter VI, Section 8]{LangAlgebra}).

Going beyond Kummer theory, Kneser \cite{Kneser} has shown that we have an equality in \ref{maininequality} provided that the following conditions hold: $n$ is not divisible by the characteristic of $K$; for every odd prime $p$ we have $\zeta_p\in K^\times$ or $\zeta_p\notin GK^\times$; we have $\zeta_4\in K^\times$ or $1\pm \zeta_4\notin G K^\times$. The second condition is motivated by our example above, while for the last condition consider $K=\mathbb Q$ and suppose that $G$ is generated by $1+\zeta_4$ (which is a fourth root of $-4$): the ratio degree/index equals $1/2$.

For any multiplicative subgroup $H$ of $\overline{K}^\times$ and for any positive integer $m$ we call $\mu_m(H)$ the group consisting of the roots of unity of order dividing $m$ that are contained in $H$. Suppose that the exponent of $GK^\times/K^\times$, which we call $n$, is not divisible by the characteristic of $K$. Rybowicz \cite{Rybowicz} (see Theorems \ref{thmKR} and \ref{Ryb2}) has proven that, if $n$ is a prime power, then we have 
$$\frac{[K(G):K]}{|G^nK^{\times n}:K^{\times n}|}=\delta \cdot   [K(\mu_n(GK^\times)):K]$$
where $\delta\in \{1, \frac{1}{2}\}$ and where $\delta=\frac{1}{2}$ can only hold if $n$ is a power of $2$. As discussed in Section \ref{sec:KneserKummer}, this result by Rybowicz builds on Kneser's theorem and also on work by Hasse and Schinzel. 

There is a vast literature on radical extensions. We mention for example also \cite{Albu} by Albu, \cite{BarreraVelez} by Barrera Mora and Vélez, and \cite{Koch2} by 
Halter-Koch. In \cite{Lenstra}, Lenstra investigated the entanglement by introducing the entanglement group: this group was studied also by Palenstijn in his PhD thesis \cite{Palenstijn} and by the author with Sgobba and Tronto \cite{MAMA}. The entanglement has also been studied by Lenstra, Moree and Stevenhagen in \cite{LMS}. Recently, the author with Chan, Pajaziti, and Perissinotto established further results on the entanglement, see \cite{EntRad}.

In this paper we build on the result by Rybowicz to study the ratio 
$$\frac{[K(G):K]}{|GK^{\times}:K^{\times}|}\,.$$
We are also able to remove the assumption that $n$ is a prime power: for us the finite exponent of $GK^\times/K^\times$ can be any positive integer that is not divisible by the characteristic of $K$. The following remark shows that our condition on $n$ is in fact not restrictive.

\begin{rem}\label{rem-char}
If $q:=\ch(K)$ and $q\mid n$ (which implies $q\neq 0$), the extension $K(G)/K(G^{q^{v_q(n)}})$ is purely inseparable and with degree the $q$-part of $|GK^\times: K|$ (see \cite[Corollary 9.2, Chapter VI]{LangAlgebra}). We deduce that
$$\frac{[K(G):K]}{|GK^\times:K^\times|}=\frac{[K(G^{q^{v_q(n)}}):K]}{|G^{q^{v_q(n)}}K^\times:K^\times|}$$
so we have reduced to the case where $n$ is not divisible by the characteristic of $K$.
\end{rem}

The following remark shows that the entanglement that we have for a general $n$ cannot be explained just by the entanglement that we have when $n$ is a prime power.

\begin{rem}
Let $n=\prod_p p^v$, where $p$ varies among the prime divisors of $n$. Defining $G_p:=G^{n/p^v}$, we have 
$$|GK^\times:K^\times|=\prod_p |G_pK^\times:K^\times|\,.$$ We deduce that 
$$\frac{[K(G):K]}{|GK^\times:K^\times|}\leq \prod_p \frac{[K(G_p):K]}{|G_pK^\times:K^\times|}$$
and that the equality holds if and only if the fields $K(G_p)$ are linearly disjoint over $K$. In fact, we prove (see Proposition \ref{prop-split}, in view of Remark \ref{rem-char}) that 
$$[K(G):K]\quad \text{divides}\quad \prod_p [K(G_p):K]\,.$$
\end{rem}

We denote by $z$ the product of the odd prime divisors $p$ of $n$ such that $\zeta_p\notin K^\times$ and $\zeta_p\in GK^\times$.
The main result of this paper is the following.

{\color{black}
\begin{thm}\label{mainthm}
Suppose that $\ch(K)\nmid n$. If $n$ is odd, then we have
\begin{equation}\label{oddratio}
\frac{[K(G):K]}{|GK^\times:K^\times|}=\frac{[K(\zeta_{z}):K]}{|G K^\times\cap K(\zeta_{z}): K^\times|}\,.
\end{equation}
If $n$ is even, we write $n=2^f n'$ where $f$ is a positive integer and $n'$ is an odd integer. Let $\Delta$ be the non-negative integer from Definition \ref{Delta} (which concerns the radical group $G^{n'}$ over $K(\zeta_z)$). Then we have 
$$\frac{[K(G):K]}{|GK^\times:K^\times|}=\frac{[K(\zeta_{z}):K]}{|G K^\times\cap K(\zeta_{z}): K^\times|} \cdot 2^{-\Delta}\,.$$
\end{thm}

As a special case of our result, if $n$ is a power of $2$ (hence $n'=z=1$) the ratio is $2^{-\Delta}$: if $f\leq 1$ or $\zeta_4\notin K$, then $2^{-\Delta}=1$ and otherwise we have $$2^{-\Delta}=\frac{[K(H):K]}{|H K^\times:K^\times|}$$
where $H=H'$ is as in Theorem \ref{better2} and Definition \ref{Delta}.}

From our main result we deduce the following property, which (applying Remark \ref{rem-char}) also holds in case $\ch(K)\mid n$ provided that we exclude from $z$ the prime factor $\ch(K)$.

\begin{thm}\label{divisibility}
The degree $[K(G):K]$ divides 
$$\frac{1}{z} \cdot [K(\zeta_z):K]  \cdot |GK^{\times}:K^{\times}|\,.$$
\end{thm}

In the last section of the paper we prove a result about the (eventual maximal) growth of radical extensions, see Theorem  \ref{eventual}. We fix a subgroup $\Gamma$ of $K^\times$ and for every $N>1$ consider the radical group $R_N$, which consists of all $N$-th roots of all elements of $\Gamma$. We show that, under suitable assumptions, there exists a positive integer $N_0$ (depending on $\Gamma$ and $K$) such that 
$$[K(R_N): K(R_{\gcd(N,N_0)})]=\frac{|R_NK^\times:K^\times|}{|R_{\gcd(N,N_0)}K^\times:K^\times|}\,.$$
This relation, rewritten as
\begin{equation}\label{grow}
\frac{[K(R_N):K]}{|R_NK^\times:K^\times|}=\frac{[K(R_{\gcd(N,N_0)}):K]}{|R_{\gcd(N,N_0)}K^\times:K^\times|}\,,
\end{equation}
shows that the ratio between the degree and the index can be understood by investigating the finitely many divisors of $N_0$. In view of the discussion below, \eqref{grow} also shows that  the $K$-linear relations between the radicals in $R_N$ are generated by the multiplicative relations and by the $K$-linear relations between the radicals in $R_{\gcd(N,N_0)}$. Theorem \ref{MAMA} (which is the reformulation in our setting of \cite[Theorem 1]{MAMA}) is similar in nature, and its formula is like \eqref{grow} but with further factors.

\subsection{Linear (and polynomial) relations among radicals}\label{sec:relations}
We want to describe all the polynomial relations among the radicals over $K$.
Any polynomial relation only involves finitely many radicals $r_1,\ldots, r_m$ (for some $m\geq 1$) and it amounts to a non-zero polynomial $f(x_1,\ldots, x_m)\in K[x_1,\ldots, x_m]$ such that
$f(r_1,\ldots, r_m)=0$. If, beyond the constant term, the monomials in $f(x_1,\ldots, x_m)$ have degree $1$, then we have a  $K$-linear relation among $r_1,\ldots, r_m, 1$. In fact studying all the polynomial relations among all radicals amounts to study their $K$-linear relations. This is because any finite product of radicals is a radical (and we can assign a new variable to any monomial of $f$, to see $f$ as a homogeneous polynomials of degree $1$ beyond the constant term).

Consider the radical group $G:=\langle r_1,\ldots, r_m \rangle$. In particular, $|GK^\times: K^\times|$ is finite. Each relation among radicals in $G$ that is of multiplicative nature (for example, $\sqrt{6}\cdot \sqrt{10}=2\sqrt{15}$ in $\mathbb R$)
is explained by the group structure of $GK^\times/ K^\times$. Thus, such relations are controlled by the index $|GK^\times: K^\times|$. The further polynomials relations are controlled by the degree $[K(G):K]$ because any relation not generated by  previously considered ones leads to a decrease of the degree.
Our results on $[K(G):K]/|GK^\times: K^\times|$ tell us what the degree is, so it suffices to find enough relations that provide as upper bound for the degree its actual value.

Our main result (Theorem \ref{mainthm}) then allows to completely understand the polynomial relations among radicals. We now describe polynomial relations that (combined with the multiplicative relations) generate all polynomial relations among radicals. As discussed, we may work with a finitely generated radical group $G$, and we keep the notation of Theorem \ref{mainthm}.

{\color{black}
\begin{itemize}
\item The degree $[K(\zeta_{z}):K]$ takes into account the $K$-linear relationships among the powers of $\zeta_z$. [More precisely, we should consider the ratio $$\frac{[K(\zeta_{z}):K]}{|\langle \zeta_z\rangle K^\times: K^\times|}$$ which means that in the next bullet point it would suffice to consider $K$-relations for radicals in $GK^\times \cap K(\zeta_{z})$ up to radicals in $\langle \zeta_z\rangle$.]
\item Let $m$ be the largest divisor of $n$ such that $\zeta_{m}\in K$, excluding from $m$ the prime factors $p$ such that $\zeta_{p^e}\in K$ holds for every $e\geq 1$.
The elements in $G K^\times\cap K(\zeta_{z})$
are $K$-linear combinations of powers of $\zeta_{z}$.
By Schinzel's theorem (Theorem 21) they are generated over $K^\times$ by radicals $\alpha$ of the form $\alpha=\zeta_{nm}^a \beta$ where $a$ is a positive integer and $\beta$ is a \emph{Kummer radical}, by which we mean that $\beta^e\in K^\times$ holds for some $e\geq 1$ such that $\zeta_e\in K$ (in fact, $K(\beta)/K$ is a Kummer extension). So we have to consider $K$-linear relations between powers of $\zeta_{\lcm(mn,z)}$ and single Kummer radicals.
\item The degree loss caused by the term $2^{\Delta}$ stems from $K$-linear relations among roots of unity and possibly an additional relation concerning the radical
$$1+\zeta_{2^w}\in 1K(\zeta_z)+\zeta_4K(\zeta_z)\,,$$
where $w$ is largest integer such that $\zeta_{2^w}+\zeta_{2^w}^{-1}\in K(\zeta_z)^\times$ (provided that such largest integer exists). This is a $K$-linear relation involving $1+\zeta_{2^w}$ and powers of $\zeta_{4z}$. It is applicable only if $4\mid n$ and $\zeta_4\notin K(\zeta_z)$ and $\mathrm{char}(K)=0$ and $1+\zeta_{2^w}\in G^{n'} K(\zeta_z)^\times$.
\end{itemize}
}

As a summary, in this paper we are able to prove the following general result.

{\color{black}
\begin{thm}
The polynomial relations among radicals are generated by the multiplicative relations and by $K$-linear relations of the following type: relations among roots of unity; relations among single Kummer radicals and roots of unity; a relation among roots of unity and $1+\zeta_{2^w}$ for some unique $w\geq 2$.
\end{thm}
}

\subsection*{Acknowledgments} We thank Daniel Gil-Munoz for discussions about a special case of the main theorem, Fritz Hörmann for general discussions and feedback, Zeev Rudnick for useful references and the hint to consider also the polynomial relations, and Pieter Moree for useful comments. We also sincerely thank the two referees, who provided valuable comments which led to an improvement of the paper, and also pointed out Remark \ref{referee}.

\section{Kneser's Theorem and Kummer theory}\label{sec:KneserKummer}

We let $K$ be a field and fix an algebraic closure $\overline{K}$. We  let $G$ be a group of radicals over $K$ such that the index $|G K^\times :K^\times|$ is finite. We let $n$ be the smallest positive integer such that $G^n$ is contained in $K^\times$ (thus, $n$ is the exponent of the group  $GK^\times/K^\times$ and it divides $|G K^\times :K^\times|$). In view of Remark \ref{rem-char} we always suppose that $n$ is not divisible by the characteristic of $K$.

We observe that, for the purposes of this paper, it would be not restrictive to suppose that $G$ is finitely generated. Indeed, it is possible to replace $G$ by any set of representatives of the group $G/(G\cap K^\times)\simeq GK^\times/K^\times$, which is assumed to be finite (as this change does not affect neither the index nor the degree under consideration).

For any positive integer $m$ that is not divisible by $\ch(K)$ we fix some root of unity $\zeta_m$ in $\overline{K}^\times$ of order $m$. We assume that these choices are coherent. More precisely, if $m,M$ are positive integers such that $m\mid M$ then we assume that $\zeta_M^{M/m}=\zeta_m$. 
If $H$ is a subgroup of $\overline{K}^\times$ and $m$ is a positive integer, we write $\mu_m(H)$ for the group of roots of unity in $H$ whose order divides $m$. We let $\varphi$ denote Euler's totient function.

\begin{rem}\label{remexact}
The following sequence, induced by the exponentiation by $n$, is exact:
\begin{equation}\label{exact}
1 \rightarrow \mu_n(GK^\times)K^\times/K^\times \rightarrow GK^\times/K^\times \rightarrow G^nK^{\times n}/K^{\times n} \rightarrow 1\,.
\end{equation}
Indeed, if $(ga)^n=b^n$ for some $g\in G$ and for some $a,b\in K^\times$, then $g^n\in K^{\times n}$ and hence $g\in \mu_n(GK^\times)K^\times$.
Moreover, the following sequence is exact 
\begin{equation*}\label{exactroot}
1 \rightarrow \mu_n(K^\times) \rightarrow \mu_n(GK^\times) \rightarrow \mu_n(GK^\times)K^\times/K^\times \rightarrow 1
\end{equation*}
because we have $\mu_n(K^\times)=\mu_n(GK^\times)\cap K^\times$.
We deduce that 
\begin{equation}\label{exactall}
|GK^\times:K^\times|=|G^nK^{\times n}:K^{\times n}|\cdot | \mu_n(GK^\times): \mu_n(K^\times)|\,.
\end{equation}
\end{rem}

We rely on the famous result by Martin Kneser from \cite{Kneser}:

\begin{thm}[Kneser's theorem]\label{Kneserthm}
We have 
$$[K(G):K]=|GK^\times:K^\times|$$
if the following two conditions hold: for every odd prime $p$ we have $\zeta_p\in K^\times$ or $\zeta_p\notin GK^\times$; we have $\zeta_4\in K^\times$ or $1\pm \zeta_4\notin G K^\times$. 
\end{thm}

We observe that for every $m> 1$ and for every $\alpha\in K^\times$ the radical group $\langle \zeta_m \alpha \rangle$ contains a root of unity different from $1$ if and only if $\alpha\neq \zeta_m^{-1}$ is a root of unity. Unless $K$ has positive characteristic $q$ and it is algebraic over $\mathbb F_q$, there is $\alpha\in K^\times$ that is not a root of unity. Thus, even if $\zeta_m\notin K^\times$ and  $G\cap \mu_m(\overline{K}^\times)=\{1\}$, it is possible that $GK^\times$ contains $\mu_m$.

Kneser's result covers all cases for which $[K(G):K]=|GK^\times:K^\times|$ holds. Indeed, in \cite[Chapter 2, Theorem 20]{Schinzelbook} Schinzel proved the following: 

\begin{prop}[Schinzel]\label{prop-Schinzel}
The two conditions in Kneser's theorem are necessary. 
\end{prop}

\begin{proof}[Proof (alternative proof based on Theorem \ref{mainthm}).]
 {\color{black} 
 If $z>1$, then $[K(\zeta_z):K]\leq \varphi(z)$  while 
$|G K^\times\cap K(\zeta_{z}):K^\times|$ is divisible by $z$} hence there is some prime $p\mid z$ such that the $p$-adic valuation of $[K(G):K]/|GK^\times:K^\times|$ is non-zero. Now suppose that $z=1$ but that the second condition in Kneser's theorem does not hold: we prove that the $2$-adic valuation of $[K(G):K]/|GK^\times:K^\times|$ is non-zero. To study this $2$-adic valuation, we may replace $G$ by $G^{n'}$ (hence $n$ becomes $2^f$). Then by Kneser's theorem over $L=K(\zeta_4)$ we have 
$[L(G):L]=|GL^\times:L^\times|$. The claim holds because $[K(G):K]/[L(G):L]$ divides $2$ while $|GK^\times:K^\times|/|GL^\times:L^\times|$ is a multiple of $4$, as (with the appropriate sign choice) the class of $(1\pm \zeta_4)\in L^\times$ has order $4$ in 
$GK^\times/K^\times$.
\end{proof}

\begin{rem}\label{Knesercondi}
If $\zeta_p\in GK^\times$ and $\zeta_p\notin K^\times$, then we have $p\mid n$. If (for a sign choice) $1\pm \zeta_4\in GK^\times$ and $\zeta_4\notin K^\times$, then $4\mid n$ because  
$(1\pm \zeta_4)^2=\pm 2\zeta_4$ and $(1\pm \zeta_4)^4=-4$  hence the order of $1\pm \zeta_4$ in $GK^\times/K^\times$ is $4$. 
\end{rem}

By the above remark, the two conditions in Kneser's theorem are satisfied if $\zeta_n\in K$.
In this case, the extension $K(G)/K$ is a \emph{Kummer extension} and its Galois group is abelian of exponent dividing $n$:

\begin{thm}[Kummer theory]
    If $\zeta_n\in K$, the groups $\Gal(K(G)/K)$ and $GK^\times/K^\times$ and $G^nK^{\times n}/K^{\times n}$ are isomorphic. In particular, we have
    \begin{equation}\label{all} [K(G):K]=|GK^\times:K^\times|=|G^nK^{\times n}:K^{\times n}|\,.
    \end{equation}
\end{thm}
\begin{proof}
The isomorphism between the former and latter group is one of the main results in Kummer theory (see \cite[Theorem 8.1, Chapter VI]{LangAlgebra}). The isomorphism between the second and the third group is a consequence of \eqref{exact} because $\mu_n(GK^\times)=\mu_n(K^\times)=\langle \zeta_n\rangle$.
\end{proof}

If $K$ and $G$ consist of real numbers, then $\mu_n(GK^\times)=\mu_n(K^\times)=\{\pm 1\}$ and the two conditions in Kneser's theorem are satisfied. Hence, \eqref{all} holds (see also  \cite[Theorem 2.2]{Rybowicz}).

\section{The case where $n$ is an odd prime power}

We suppose that $n$ is the power of an odd prime number $p$ (thus, the characteristic of $K$ is different from $p$). 
We rely on the following result, which combines Theorem \ref{Kneserthm} (in view of  Remark \ref{Knesercondi}) and \cite[Theorem 2.3]{Rybowicz}. To be precise, \cite[Theorem 2.3]{Rybowicz} is stated for characteristic zero (there is a remark in the paper that most results should extend to separable extensions in positive characteristic) however the proof goes through also in positive characteristic: this can easily be checked as the proof is self-contained up to Kneser's theorem.

\begin{thm}[Kneser - Rybowicz]\label{thmKR}
We have
$$\frac{[K(G):K]}{|G^nK^{\times n}:K^{\times n}|}=  [K(\mu_{n}(GK^\times)):K]\,.$$
Moreover, if $\zeta_p\notin GK^\times$ or $\zeta_p\in K^\times$, then we have 
$$[K(G):K]= |GK^{\times}:K^{\times}|\,.$$
\end{thm}

We deduce that, if $\zeta_p\notin GK^\times$ or $\zeta_p\in K^\times$, then the degree $[K(G):K]$ is a power of $p$ while in the remaining case it is a power of $p$ times $[K(\zeta_p):K]$.

\begin{cor}\label{corbetterp}
If $H:=\mu_{n}(GK^\times)$, then we have
$$\frac{[K(G):K]}
{|GK^{\times}:K^{\times}|}=\frac{[K(H):K]}
{|HK^{\times}:K^{\times}|}\,.$$
Let $m_0$ be the largest positive integer such that $\zeta_{p^{m_0}}\in K(\zeta_p)^\times$ (or $\infty$, if no such largest integer exists). Then we have
$$\frac{[K(G):K]}
{|GK^{\times}:K^{\times}|}=\begin{cases} 1 & \text{if $\zeta_p\in K^\times$ or $\zeta_p\notin GK^\times$}\\
[K(\zeta_p):K] \cdot p^{-{\min(m_0,m)}} & \text{otherwise}\,.
\end{cases}
$$
Moreover, if $\zeta_p\notin K^\times$ and $\zeta_p\in GK^\times$, then we have  $H=\langle \zeta_{p^m}\rangle$ for some positive integer $m$.
\end{cor}
\begin{proof}
    Combining Theorem \ref{thmKR} and  \eqref{exactall} we get
$$\frac{[K(G):K]}
{|GK^{\times}:K^{\times}|}=\frac{[K(\mu_{n}(GK^\times)):K]}{| \mu_{n}(GK^\times): \mu_{n}(K^\times)|}\,.$$
From $H\cap K^{\times}=\mu_{n}(K^\times)$ we deduce that $|H: \mu_{n}(K^\times)|=|HK^{\times}:K^{\times}|$.
By Theorem \ref{thmKR} we are left to deal with the case $\zeta_p\notin K^\times$ and $\zeta_p\in GK^\times$. We may conclude because $|H: \mu_{n}(K^\times)|=p^m$ while $[K(H):K]=[K(\zeta_p):K]\cdot p^{\max(m-m_0,0)}$.
\end{proof}

\begin{rem}
We observe that we may have $m_0=\infty$ even if $\zeta_p\notin K$. For example, let $F$ be a prime field with a characteristic that is not $p$ and such that $\zeta_p\notin F$.  We then define $K=F(\{\zeta_{p^t}+\zeta_{p^t}^{-1}:t\geq 1\})$, noticing that the map $\zeta_{p^t}\mapsto \zeta_{p^t}^{-1}$ is the identity on $K$ but not on $\zeta_p$ hence $\zeta_p\notin K$.

The requested properties are preserved if we enlarge $K$ without changing $K\cap F(\zeta_{p^\infty})$.
\end{rem}

As explained in Section \ref{sec:relations}, we can make use of Corollary \ref{corbetterp} to determine $K$-linear relations among radicals that, together with the multiplicative relations, generate all the $K$-linear relations. If $n$ is an odd prime power, all  entanglement stems from $K$-linear relations among roots of unity. The $K$-linear relation
$$1+\zeta_p+\zeta_p^2+\cdots +\zeta_p^{p-1}=0$$
has to be counted only if $\zeta_p\notin K^\times$, and it is only relevant if $\zeta_p\in GK^\times$.
If $d_p:=[K(\zeta_p):K]$, the minimal polynomial of $\zeta_p$ gives a $K$-linear relation among $1,\zeta_p,\ldots,\zeta_p^{d_p}$, and all the powers $\zeta_p^i$ for $i\geq d_p$ are $K$-linear combinations of the roots of unity $1,\zeta_p,\ldots, \zeta_p^{d_p-1}$.
We also have 
$$(\zeta_{p^m})^{p^{j+\max(m-m_0,0)}}\in 1K+\zeta_pK+\cdots +\zeta_p^{d_p-1}K \qquad \text{for}\quad j\geq 0$$
because any element in $K(\zeta_p)$ is of this form.
These $K$-linear relations already explain the degree of $K(G)/K$.

\section{The case where $n$ is a power of $2$}\label{sec:better2}

Let $n=2^f$ for some positive integer $f$ (so we assume that $\ch(K)\neq 2$). \emph{We suppose that $f\geq 2$ and $\zeta_4\notin K^\times$} (else, we already know that $[K(G):K]=|GK^\times:K^\times|$ by Theorem \ref{Kneserthm} and Remark \ref{Knesercondi}).
For every positive integer $t$ we write $\xi_{2^t}=\zeta_{2^t}+\zeta_{2^t}^{-1}$. Moreover, we let $w$ be the largest integer such that $\xi_{2^w}\in K^\times$, or we set $w=\infty$ if no such largest integer exists.

The following lemma is due to Hasse \cite[Satz 2]{Hasse}, and it also holds for $f=1$ (see \cite[Lemma 2]{Schinzel-ab} by Schinzel).

\begin{lem}[Hasse and Schinzel]\label{lemminoSchinzel}
 The kernel of the map
$$K^{\times}/K^{\times n}\rightarrow K^{\times}K(\zeta_4)^{\times n}/K(\zeta_4)^{\times n}
$$
induced by the inclusion is generated by the class of the following element:
$$a=\begin{cases} -1 & \text{if $w>f$}\\
-\xi_{2^{w+1}}^n & \text{if $w=f$} \\
\xi_{2^{w+1}}^n & \text{if $w<f$}\,. \\
\end{cases}$$
\end{lem}

Notice that $\xi_{2^{w+1}}^n=(\xi_{2^w}+2)^{\frac{n}{2}}$ because $\xi_{2^{w+1}}^2 = \xi_{2^w}+2$. Since $\xi_{2^w}\in K^{\times}$, we always have $a^2\in K^{\times n}$. However, the class of $a$ modulo $K^{\times n}$ may have order $1$ or $2$. 

\begin{lem}\label{orders}
With the notation of Lemma \ref{lemminoSchinzel}, we have $a\notin K^{\times n}$ if and only if 
$w\geq f$ or $\ch(K)=0$ and $K\cap \Q(\zeta_{2^\infty})$ is totally real.
\end{lem}
\begin{proof}
We can easily settle the case $w>f$ because $(-1)\notin K^{\times n}$ as $\zeta_4\notin K^\times$. So now suppose that $w\leq f$.
The condition $a\in K^{\times n}$ means that $(\xi_{2^{w}}+2)\gamma \in K^{\times 2}$, where $\gamma$ is a root of unity of order dividing $n/2$ for $w<f$ and of order $n$ for $w=f$. Since $\gamma\in K^\times$, we must have $\gamma\in \{\pm 1\}$. We cannot have $\gamma=1$ because $\xi_{2^{w+1}}\notin K$, so now suppose that $\gamma=-1$. 

Suppose first of all that $K$ has odd characteristic $p$, and let $F=
\mathbb F_p(\xi^2_{2^{w+1}})$. Then, $F$ is a finite field contained in $K$. Therefore, $\zeta_4\notin F$, and the product of two non-squares
in $F$ is a square (by Kummer theory, as $F(\sqrt{b})=F(\zeta_4)$ holds if $b\in F^\times\setminus F^{\times  2}$).
Hence $-(\xi_{2^{w}}+2)=(-1)\cdot (\xi_{2^{w+1}}^2) \in F^{\times 2} \subseteq  K^{\times 2}$ in odd characteristic.

In characteristic $0$, the square roots of $-(\xi_{2^{w}}+2)$ are in $\Q(\zeta_{2^\infty})$ and not totally real. They cannot be in $K^\times$ if $K\cap \Q(\zeta_{2^\infty})$ is totally real. In the remaining case, $\zeta_4$ and $\xi_{2^{w+1}}$ generate the same quadratic extension of $K\cap \Q(\zeta_{2^\infty})$ hence of $K$, so by Kummer theory $-(\xi_{2^{w}}+2)$ is a square in $K^\times$.
\end{proof}

We rely on the following result, which is \cite[Theorem 2.4]{Rybowicz}, restated thanks to Lemma \ref{orders}. Similarly to Theorem \ref{thmKR}, the original result is stated in characteristic zero, but the proof also goes through in positive characteristic.

\begin{thm}[Rybowicz]\label{Ryb2}
Let $n=2^f$ for some integer $f\geq 1$. We have 
$$\frac{[K(G):K]}{|G^nK^{\times n}:K^{\times n}|}=\delta_G \cdot [K(\mu_n(GK^\times)):K]$$
where $\delta_G\in \{1,\frac{1}{2}\}$.
Let $a$ be as in Lemma \ref{lemminoSchinzel}.
We have $\delta_G=\frac{1}{2}$ if and only if $a\in G^nK^{\times n}$ and, in the case $w<f$, additionally $1+\zeta_{2^w}\in GK^\times$ and $\ch(K)=0$ and $K\cap \Q(\zeta_{2^\infty})$ is totally real.
\end{thm}

\begin{thm}\label{better2}
Let $\delta_G$ be as in Theorem \ref{Ryb2} and set
$$H=\begin{cases} \mu_{2n}(GK^\times) & \text{if $w>f$}\\
\mu_{n}(GK^\times) & \text{if $w=f$ and $-\xi_{2^{w+1}}^n\notin G^nK^{\times n}$}\\
\langle 1+\zeta_{2^w}, \zeta_n\rangle & \text{if $w=f$ and $-\xi_{2^{w+1}}^n\in G^nK^{\times n}$}\\
\mu_{n}(GK^\times) & \text{if $w<f$ and $\delta_G=1$}\\
\langle 1+\zeta_{2^w} \rangle\mu_{n}(GK^\times)& \text{if $w<f$ and $\delta_G=1/2$}\,.\\
\end{cases}$$
Then $H$ is a subgroup of $GK^\times$ and we have
$$\frac{[K(G):K]}
{|GK^{\times}:K^{\times}|}= \frac{[K(H):K]}
{|HK^{\times}:K^{\times}|}\,.$$
\end{thm}
\begin{proof}
For $w=f$, we first prove the claim that $H=\langle 1+\zeta_{2^w}, \zeta_n\rangle \cap GK^\times$. Observe that $\zeta_{2^{w+1}}\xi_{2^{w+1}}=1+\zeta_{2^w}$. If $-\xi_{2^{w+1}}^n\notin G^nK^{\times n}$, there is no integer $i$ such that $(1+\zeta_{2^w})\zeta_n^i\in GK^\times$ and hence $\langle 1+\zeta_{2^w}, \zeta_n\rangle \cap GK^\times=\mu_n(GK^\times)$. Else, fix $i$ such that $(1+\zeta_{2^w})\zeta_n^i\in GK^\times$. Considering that $\xi_{2^{w+1}}^2\in K^\times$, we deduce that $\zeta_n\in GK^\times$ hence $1+\zeta_{2^w}\in GK^\times$ and the claim is proven.

Note that in all cases $H$ is a subgroup of $GK^\times$ such that $\mu_n(HK^\times)=\mu_n(GK^\times)$. We now prove that $\delta_G=\delta_{H}$.
If $w>f$, this is because $-1\in G^nK^{\times n}$ is equivalent to $\zeta_{2n}\in GK^{\times}$.
If $w=f$, this is because $-\xi_{2^{w+1}}^n\in G^nK^{\times n}$ is equivalent to $(1+\zeta_{2^w})\in \mu_n(\overline{K}^\times)\cdot  GK^{\times}$ and we have $H=\langle 1+\zeta_{2^w}, \zeta_n\rangle \cap GK^\times$. If $w<f$, we observe that $\delta_G=1$ implies $\delta_{H}=1$, so suppose that $\delta_G=\frac{1}{2}$. We clearly have $1+\zeta_{2^w}\in HK^{\times}$, so we are left to prove that $\xi_{2^{w+1}}^n\in H^nK^{\times n}$. This is equivalent to $\xi_{2^{w+1}}\in \langle 1+\zeta_{2^w},\zeta_n \rangle K^{\times n}$, which holds because $\xi_{2^{w+1}}=\zeta_{2^{w+1}}^{-1}(1+\zeta_{2^w})$ and $\zeta_{2^{w+1}}$ is a power of $\zeta_{n}$.

In view of Remark \ref{remexact}, from Theorem \ref{Ryb2} (applied to $G$ and to $H$), as $\tilde{H}:=\mu_n(GK^\times)=\mu_n(HK^\times)$ and $\delta_G=\delta_H$, we have 
$$\frac{[K(G):K]}
{|GK^{\times}:K^{\times}|}=\delta_G \cdot \frac{[K(\tilde{H}):K]}{|\tilde{H}:\mu_n(K^{\times})|}=\frac{[K(H):K]}
{|HK^{\times}:K^{\times}|}\,.$$
\end{proof}

\begin{exa}
    All five cases in Theorem \ref{better2} do occur. For the first case we may choose $K=\mathbb Q(\xi_{2^w})$ with $w\geq 3$ and  $G=\langle \sqrt[n]{5} \rangle$ with $2\leq f<w$. Now let $K=\mathbb Q$ (thus, $w=2$). In the second and third case we have $n=4$ and the condition to check is whether $-4\in G^4 \mathbb Q^{\times 4}$. So for the second (respectively, third) case we may take $G$ to be $\langle \sqrt[4]{5} \rangle$ (respectively, $\langle \zeta_8\sqrt{2} \rangle$). Finally, let $n=8$ hence $a=16$. For the fourth case we may take $G=\langle \sqrt[8]{5} \rangle$ as $16\notin \langle 5 \rangle\mathbb Q^{\times 8}$. For the fifth case we may take $G=\langle \zeta_8\sqrt{2}\rangle$ as $16\in \langle 16 \rangle\mathbb Q^{\times 8}$ and $1+\zeta_4=\zeta_8\sqrt{2}$.
\end{exa}

\begin{lem}\label{lembetter2}
With the notation of Theorem \ref{better2}, we let $m$ be the largest positive integer such that $\zeta_{2^m}\in H$. If $m=1$ then we have 
$\frac{[K(H):K]}{|HK^\times:K^\times|}=1$, while if $m\geq 2$ then we have 
$$\frac{[K(H):K]}{|HK^\times:K^\times|}=\begin{cases} 
2^{2-m} & \text{if $w>f$ or if $w=f$ and $-\xi_{2^{w+1}}^n\notin G^nK^{\times n}$}\\
2^{1-f} & \text{if $w=f$ and $-\xi_{2^{w+1}}^n\in G^nK^{\times n}$}\\
2^{2-\min(w',m)} & \text{if $w<f$, $\delta_G=1$}\\
2^{1-\min(w',m)} & \text{if $w<f$, $\delta_G=1/2$}\end{cases}$$
where, if $w$ is finite, $w'$ is the largest positive integer such that $\zeta_{2^{w'}}\in K(\zeta_4)^\times$.
In the last case we have $m=\max(w, \overline{m})$, where $\overline{m}$ is the largest integer such that $\zeta_{2^{\overline{m}}}\in \mu_n(GK^\times)$.
\end{lem}

We observe the following: in characteristic $0$, we have $w'=w$ or $w'=w+1$ and the latter case holds if and only if $\ch(K)=0$ and $K\cap \Q(\zeta_{{2^\infty}})$ is not totally real; in odd characteristic $p$ (thus $p\equiv 3 \bmod 4$ because $\zeta_4\notin K$), $w'$ is the $2$-adic valuation of $p^2-1$. 

\begin{proof}
If $m=1$, then  $[K(H):K]=|HK^\times:K^\times|$ by Theorem \ref{Kneserthm} and Remark \ref{Knesercondi} so suppose that $m\geq 2$.

We remark that $m\leq w$ if $w>f$ (because  $m\leq f+1$) or if $w=f$ and $-\xi_{2^{w+1}}^n\notin G^nK^{\times n}$. In these cases, we have $[K(H):K]=2$ and $|HK^\times:K^\times|=|H:\mu_{2n}(K^\times)|=2^{m-1}$.

If $w=f$ and $-\xi_{2^{w+1}}^n\in G^nK^{\times n}$, then $K(H)=K(\zeta_4)$.
We conclude because $(1+\zeta_{2^w})^2=\zeta_{2^w}(2+\xi_{2^w})$ and $\zeta_{2^w}$ are in the same class modulo $K^\times$ and hence   
$$|HK^\times: K^\times|=2|\langle \zeta_n\rangle K^\times: K^\times|=n\,.$$

Finally suppose that $w<f$. 
Since $1+\zeta_{2^w}\in K(\zeta_4)^\times$, we have 
$$[K(H):K]=2^{1+\max(m-w',0)}\,.$$

If $\delta_G=1$, we may conclude because we have $|HK^\times:K^\times|=2^{m-1}$.
If $\delta_G=1/2$ (in particular, $\ch(K)=0$ and $K\cap \Q(\zeta_{{2^\infty}})$ is  totally real), recall that $(1+\zeta_{2^w})^2\in \zeta_{2^w} K^\times\setminus K^\times $. By Lemma \ref{orders} we know that $(1+\zeta_{2^w})^n=\xi_{2^{w+1}}^{n}\notin K^{\times n}$ so there is no integer $i$ such that 
$(1+\zeta_{2^w})\zeta_{n}^i\in K^\times$.
We deduce that 
$$|HK^\times: K^\times|=2 |\langle \zeta_{2^w}\rangle \mu_n(GK^\times)K^\times: K^\times|\,.$$
We are left to show that $|HK^\times: K^\times|=2^m$ and we do so by proving that $m=\max(w,\overline{m})$. 
For $\overline{m}\geq w$, $H$ is contained in $K(\zeta_{2^{\overline{m}}})^\times$ and we conclude because this group does not contain $\zeta_{2^{\overline{m}+1}}$.
For $\overline{m}<w$, $H$ is contained in $K(\zeta_{4})^\times$ and we conclude because $\zeta_{2^{w+1}}\notin K(\zeta_{4})^\times$.
\end{proof}

We make use of Theorem \ref{better2} and Lemma \ref{lembetter2} (as described in Section \ref{sec:relations}) to find $K$-linear relations among radicals which, together with the multiplicative relations, generate all $K$-linear relations (in case $n$ is a power of $2$).
Remark that $K(\zeta_4)=1K+\zeta_4 K$. 

\begin{itemize}
\item If $w>f$ or if $w=f$ and $-\xi_{2^{w+1}}^n\notin G^n K^{\times n}$, we only have entanglement if $m\geq 3$. Since $m \leq f+1$ and $m\leq f$ for $w=f$, it suffices to consider the $K$-linear relations  expressing 
$$\zeta_{2^3},\ldots, \zeta_{2^m}\in 1 K+\zeta_4K\,.$$
\item If $w=f$ and $-\xi_{2n}^n\in G^n K^{\times n}$, there is also an additional entanglement (as there is the loss of a factor $2$ in the degree $[K(G):K]$) which is due to $1+\zeta_{2^w}\in GK^\times \cap K(\zeta_4)^\times$, and it is expressed by the $K$-linear relation
\begin{equation}\label{eq:specialent}
1+\zeta_{2^w}\in 1 K+ \zeta_4 K\,.    
\end{equation}
\item Finally, suppose that $w<f$. If $\delta_G=1$, then the entanglement is similarly due to
\begin{equation}\label{eq:normalent}
\zeta_{2^3},\ldots \zeta_{2^{m'}}\in 1K +\zeta_4 K
\end{equation}
where $m'$ is the largest positive integer less or equal to $m$ such that $\zeta_{2^{m'}}\in K(\zeta_4)^\times$.
If $\delta_G=1/2$, the  entanglement is similarly explained by \eqref{eq:normalent} and \eqref{eq:specialent}.
\end{itemize}

\section{The general case}

Let $n$ be a positive integer that is not divisible by the characteristic of $K$. If $n=1$, then we have 
$$[K(G):K]=|GK^\times:K^\times|=1$$
so we suppose that $n\geq 2$ and write 
$n=\prod_{p} p^{v_p}$ for the prime factorization of $n$, where $p$ varies among the prime divisors of $n$. Let $z$ be the product of the odd primes $p$ such that $\zeta_p\notin K^\times$ and $\zeta_p\in GK^\times$. We observe that $\zeta_z\in K(G)$. 
We set $n_p:=p^{v_p}$ and $G_p=G^{n/n_p}$. In this way, $n_p$ is the smallest positive integer such that $G_p^{n_p}\subseteq K^\times$. Since $n_p$ is a prime power, we may apply the results in the previous sections to study $G_p$.
Note that we have
\begin{equation}\label{ics}
|GK^\times:K^\times|=\prod_p |G_pK^\times:K^\times|
\end{equation}
and the same holds if we replace $K$ by a finite extension. 

\begin{prop}\label{prop-split}
We have
\begin{equation}\label{eq-split}
[K(G):K(\zeta_z)]=\prod_p [K(\zeta_z,G_p):K(\zeta_z)]
\end{equation}
and 
$$[K(G):K] \quad \text{divides} \quad \prod_p [K(G_p):K]\,.$$
\end{prop}
\begin{proof}
To prove \eqref{eq-split} we show that the fields $K(\zeta_z,G_p)$, whose compositum is $K(G)$, are linearly disjoint over $K(\zeta_z)$. It suffices to prove that the degree 
$[K(\zeta_z,G_p):K(\zeta_z)]$ is a power of $p$. This holds by Corollary \ref{corbetterp} (for $p$ odd), by Theorem \ref{Kneserthm} (for $p=2$ and $\zeta_4\in K^\times$ or $4\nmid n$) and by 
Theorem \ref{better2} and Lemma \ref{lembetter2} (in the remaining case).

The second assertion follows from \eqref{eq-split}. Indeed, consider that 
 the Galois extension $K(\zeta_z)/K$ is the compositum of its Galois subextensions $K(\zeta_p)/K$ hence 
 $$[K(\zeta_z):K] \quad \text{divides} \quad \prod_{p} [K(\zeta_p):K]\,.$$
Moreover, since $K(\zeta_z)/K(\zeta_p)$ is Galois and applying \cite[Theorem 1.12, Chapter VI]{LangAlgebra}, 
$$[K(\zeta_z,G_p):K(\zeta_z)] \quad \text{divides} \quad [K(\zeta_p, G_p):K(\zeta_p)]\,.$$ 
\end{proof}

The following result is \cite[Theorem 2]{Schinzel-ab}:

\begin{thm}[Schinzel's theorem on abelian radical extensions]\label{Schinzel-abelian}
Let $n\geq 1$ be not divisible by $\ch(K)$. 
If $\alpha\in K^\times$, the extension $K(\zeta_{n}, \sqrt[n]{\alpha})/K$ is abelian if and only if $\alpha^{m}=\beta^{n}$ holds for some $\beta\in K^{\times}$ and for some $m\mid n$ such that $\zeta_{m}\in K$.
\end{thm}

{\color{black}\begin{proof}[Proof of Theorem \ref{mainthm} if $n$ is odd]
By \eqref{ics} and \eqref{eq-split} we can write
$$\frac{[K(G):K]}{|GK^\times:K^\times|}=[K(\zeta_z):K]\cdot \prod_{p\mid n}
    \frac{[K(\zeta_z,G_p):K(\zeta_z)]}{|G_pK^\times:K^\times|}\,.
    $$    
By Theorem \ref{Kneserthm} (in view of Remark \ref{Knesercondi}) we  have 
$$[K(\zeta_z,G_p):K(\zeta_z)]=|G_pK(\zeta_z)^\times:K(\zeta_z)^\times|=\frac{|G_pK^\times:K^\times|}{|G_p K^\times\cap K(\zeta_z): K^\times|}\,.
    $$
We may then conclude remarking that 
$$|G K^\times\cap K(\zeta_{z}):K^\times|=\prod_{p\mid n} |G_p K^\times\cap K(\zeta_{z}):K^\times|\,.$$ 
\end{proof}
}

\begin{defi}\label{Delta}
We set $\Delta=0$ if $\zeta_4\in K^\times$ or $4\nmid n$. Otherwise, we let 
$H'$ be the group $H$ from Theorem \ref{better2} and Lemma \ref{lembetter2} for the radical group $G^{n'}$ over $K(\zeta_z)$ and set 
\begin{equation}\label{magicDelta}
2^{-\Delta}:=\frac{[K(\zeta_z, H'):K(\zeta_z)]}{|H'K(\zeta_z)^\times:K(\zeta_z)^\times|}\,.
\end{equation}
\end{defi}

{\color{black}
\begin{proof}[Proof of Theorem \ref{mainthm} if $n$ is even]
Call $G_{P}=\prod_{p\mid n, p\neq 2} G_p$. Remarking that $\zeta_z\in G_P$, we can write     
$$\frac{[K(G):K]}{|GK^\times:K^\times|}=[K(\zeta_z):K]\cdot 
\frac{[K(G_{P}):K(\zeta_z)]}{|G_{P}K^\times:K^\times|} \cdot
    \frac{[K(\zeta_z,G_2):K(\zeta_z)]}{|G_2 K^\times:K^\times|}\,.$$
By the odd case of Theorem \ref{mainthm} we have
$$\frac{[K(G_{P}):K]}{|G_{P}K^\times:K^\times|}=\frac{[K(\zeta_z):K]}{|G_P K^\times\cap K(\zeta_{z}): K^\times|}\,.$$
We are left to prove that
$$\frac{[K(\zeta_z,G_2):K(\zeta_z)]}{|G_2 K^\times:K^\times|}=\frac{2^{-\Delta}}{|G_2 K^\times\cap K(\zeta_z): K^\times|}\,.$$
So it suffices to show that 
$$\frac{[K(\zeta_z,G_2):K(\zeta_z)]}{|G_2 K(\zeta_z)^\times:K(\zeta_z)^\times|}=2^{-\Delta}\,,$$
which is a consequence of Theorem \ref{better2} and \eqref{magicDelta} (or of Theorem \ref{Kneserthm} if $\zeta_4\in K^\times$ or $4\nmid n$).
\end{proof}
}

{\color{black}We let $\sqrt{K^\times}$ be the group of all radicals over $K$ whose square is in $K^\times$.}

\begin{exa}
Let us verify the statement of Theorem \ref{mainthm} for $K=\mathbb Q(\sqrt{5})$ and $G=\langle \zeta_{n}, \sqrt[n]{g}\rangle$, where $n>1$ and $g=2\sqrt{5}(1-\sqrt{5})$. 

{\color{black} 
By Schinzel's theorem, as $\zeta_4\notin K$ and $\zeta_p\notin K$ for $p$ odd: if $n$ is odd, we have  
$$G K^\times\cap K(\zeta_{z})=\mu_{n}(G K^\times)\cap K(\zeta_{z})$$
and, if $n$ is even, we have 
$$G^{n'} K^\times\cap K(\zeta_{z})=\big(\mu_{2^{f+1}}(G K^\times)\cdot (G K^\times\cap\sqrt{K^\times})\big)\cap K(\zeta_{z})\,.$$
}
As shown in \cite[Example 42]{MAMA}, we have 
$[K(G):K]=n\varphi(n)/d$ (where $d=4$ if $10\mid n$ and $d=2$ if $\gcd(10,n)=5$ and $d=1$ otherwise). We have $|GK^\times:K^\times|=n^2$. Let $p$ vary among the prime divisors of $n$. If $n$ is odd, then we have $$[K(\zeta_z):K]=\frac{1}{d}\prod_p (p-1)=\frac{\varphi(n) z}{nd}$$
and $\mu_{n}(G K^\times)\cap K(\zeta_{z})^\times= \mu_{z}(\overline{K}^\times)$ while $\mu_{n}(K^\times)=\{1\}$ hence
$$\frac{[K(G):K]}{|GK^\times:K^\times|}=\frac{\varphi(n)}{nd} = \frac{[K(\zeta_{z}):K]}{|\mu_{n}(G K^\times)\cap K(\zeta_{z})^\times: \mu_{n}(K^\times)|}\,.$$
If $n$ is even, then we have   
$$\big(\mu_{2^{f+1}}(G K^\times)\cdot (G K^\times\cap\sqrt{K^\times})\big)\cap K(\zeta_z)^\times=\begin{cases} \{\pm 1\}{\color{black} K^\times} & \text{if $5\nmid n$}\\ \langle \sqrt{g} \rangle K^\times & \text{otherwise}\,.\end{cases}$$
We may then recover the formula of Theorem \ref{mainthm} because $\Delta =0$ (as $H'=\mu_{2^f}(\overline{K}^\times)$ for $f\geq 2$).
\end{exa}

\begin{proof}[Proof of Theorem \ref{divisibility}]
Equivalently, we prove that   
$[K(G):K(\zeta_z)]$ divides 
$\frac{1}{z} \cdot |GK^{\times}:K^{\times}|$.
Recall from \eqref{eq-split} that
$$[K(G):K(\zeta_z)]=\prod_{p\mid n} [K(\zeta_z, G_p):K(\zeta_z)]$$
and write
$$\frac{1}{z} \cdot |GK^{\times}:K^{\times}|=\left(\prod_{p\mid z}\frac{1}{p} \cdot |G_pK^{\times}:K^{\times}|\right) \left(\prod_{p\mid n, p\nmid z} |G_pK^{\times}:K^{\times}|\right)\,.$$
For $p\mid z$ the degree 
$[K(\zeta_z, G_p):K(\zeta_z)]$ divides 
{\color{black}$\frac{1}{p} \cdot |G_pK^{\times}:K^{\times}|$} by Corollary \ref{corbetterp}. If $p\neq 2$ and  $p\nmid z$, or if $p=2$ and $\zeta_4\in K^\times$ or $4\nmid n$ we have 
\begin{equation}\label{indexformula}
[K(\zeta_z, G_p):K(\zeta_z)]=|G_pK(\zeta_z)^{\times}:K(\zeta_z)^{\times}|
\end{equation}
by Theorem \ref{Kneserthm} (in view of Remark \ref{Knesercondi}) and the index in \eqref{indexformula} divides $|G_pK^{\times}:K^{\times}|$. For $p=2$, $\zeta_4\notin K^\times$ and $4\mid n$ the degree $[K(\zeta_z, G_2):K(\zeta_z)]$ divides $|G_2 K(\zeta_z)^{\times}:K(\zeta_z)^{\times}|$ by Theorem \ref{better2} and \eqref{magicDelta}.
\end{proof}

\begin{rem}\label{referee}
As pointed out by one of the referees, an alternative proof for Theorem \ref{divisibility} for $n$ odd can be obtained by writing
$$\frac{|GK^\times:K^\times|}{|K(G):K(\zeta_z)|\cdot z}=\frac{|\mu_n(GK^\times)\cap K(\zeta_z)^\times:\mu_n(K^\times)|}{z}$$
and observing that the right-hand side is an integer. The above formula also provides an explicit upper bound for the left-hand side.
\end{rem}

We set $\mu_\infty=\cup_{m\geq 1}\mu_m$. We conclude by proving two results on the growth of certain radical extensions.

\begin{thm}\label{eventual}
For every positive integer $N$ let $R_N$ be a subgroup of $\overline{K}^\times$ such that the index $|R_NK^\times:K^\times|$ divides $N^c$ for some constant $c$, $R_1\subseteq K^\times$ and such that $R_{N}^M=R_{N/M}$ holds for every $M\mid N$. Suppose that there are only finitely many primes $p$ such that $\zeta_p\notin K^\times$ and $\zeta_p\in R_N K^\times$ for some $N$, and call $z$ their product. Moreover, suppose that 
$$|\mu_{\infty}(K(\zeta_{4z})^\times):\mu_{\infty}(K^\times)|$$
is finite.
Then there exists a positive integer $N_0$ such that 
$$\frac{[K(R_N):K]}{|R_NK^\times:K^\times|}=\frac{[K(R_{\gcd(N,N_0)}):K]}{|R_{\gcd(N,N_0)}K^\times:K^\times|}\,.$$
\end{thm}
{\color{black}
\begin{proof}
Let $N_0$ be a positive integer with the property that for every $N$ we have $$R_N K^\times\cap K(\zeta_{z}) < R_{N_0} K^\times\,.$$ 
If $N$ is odd, we may conclude by Theorem \ref{mainthm} 
because for any $z'\mid z$ we have
$$R_N K^\times\cap K(\zeta_{z'}) = R_{\gcd(N, N_0)} K^\times\cap K(\zeta_{z'})\,.$$
In case $N$ is even, we have to additionally show that the number $\Delta$ does not change when replacing $N$ by $\gcd(N, N_0)$: by Definition \ref{Delta} and Theorem \ref{better2} it suffices to take the $2$-adic valuation of $N_0$ sufficiently large such that, for every $z'\mid z$ such that $\zeta_4\notin K(\zeta_{z'})$ and for every $t\geq v_2(N_0)$, the subgroup $H$ for $R_{2^{t}}$ over $K(\zeta_{z'})$ leads to the same ratio in Lemma \ref{lembetter2} (in those formulas, $w$ is finite and we will have $w<f$ and $w'<m$ if $t$ is sufficiently large).
\end{proof}
}

The following result is the reformulation in our setting of \cite[Theorem 1]{MAMA}:

\begin{thm}\label{MAMA}
Let $K$ be a number field, fix a finitely generated subgroup $\Gamma$ of $K^\times$ and for every positive integer $N$ let $R_N$ consist of all $N$-th roots of all elements of $\Gamma$. Then there exists a positive integer $N_0$ such that 
$$\frac{[K(R_N):K]}{|R_NK^\times:K^\times|}=\frac{[K(R_{\gcd(N,N_0)}):K]}{|R_{\gcd(N,N_0)}K^\times:K^\times|}\cdot \prod_{p\mid N, p\nmid N_0, \zeta_p \notin K^\times} \frac{p-1}{p}\,.$$
\end{thm}
{\color{black}
\begin{proof}
As $K$ is a number field, there is a positive even integer $Z$ such that for all positive integers $M$ coprime to $Z$ we have $[K(\zeta_M):K]=\varphi(M)$. In particular, the exponent $E$ of $\mu(K^\times)$ divides a power of $Z$. Moreover, we may choose $Z$ such that $K\cap \mathbb Q(\zeta_{\infty})\subseteq \mathbb Q(\zeta_{Z})$, so that $K(\zeta_M)$ is linearly disjoint from $K(\zeta_{N_Z})$ for any positive integer $N_Z$ which divides a power of $Z$.

Set $\tilde{R}_N:=R_{\gcd(N, Z^\infty)}$. By Theorem \ref{eventual} there is an positive integer $N_0$, which we may take to be divisible by $Z$ and to divide a power of $Z$, such that 
$$\frac{[K(\tilde{R}_N):K]}{|\tilde{R}_NK^\times:K^\times|}=\frac{[K(\tilde{R}_{\gcd(N,N_0)}):K]}{|\tilde{R}_{\gcd(N,N_0)}K^\times:K^\times|}$$
which settles the result for those $N$ which divide a power of $Z$.
To conclude, it suffices to choose $Z$ such that, writing $N=N_Z M$ where $M$ is the largest divisor of $M$ coprime to $Z$, we have
\begin{equation}\label{degree}
[K({R}_N):K]=[K({R}_{N_Z}):K]\cdot [K({R}_M):K]
\end{equation}
and
$$\frac{[K({R}_M):K]}{|{R}_MK^\times:K^\times|}=\prod_{p\mid M} \frac{p-1}{p}\,.$$

The extensions $K(R_{N_Z}, \zeta_N)$ and $K(R_{M}, \zeta_N)$ are linearly disjoint over $K(\zeta_N)$ because they have coprime degrees by Kummer theory. So we have $K(R_{N_Z})\cap K(R_{M})\subseteq K(\zeta_N)$.

As a consequence of Schinzel's theorem (Theorem \ref{Schinzel-abelian}) we have 
$$R_M K^\times \cap K(\zeta_\infty)=\langle \zeta_{M}\rangle  K^\times\,.$$
So by Kummer theory over $K(\zeta_M)$ we deduce that $K(R_{M})\cap K(\zeta_N)=K(\zeta_M)$. We claim that, up to replacing $Z$ by a multiple, we have $K(R_{N_Z})\cap K(\zeta_N) \subseteq K(\zeta_{Z^\infty})$. We deduce that \eqref{degree} holds. 

Let $M_0$ be the square-free part of $M$. Then by Theorem \ref{mainthm} applied to $R_M$ we get 
$$\frac{[K({R}_M):K]}{|{R}_MK^\times:K^\times|}=\frac{[K(\zeta_{M_0}):K]}{|R_M K^\times\cap K(\zeta_{M_0}): K^\times|}$$
so we may conclude because on the right hand side the numerator is $\varphi(M_0)$ and the denominator is $|\langle \zeta_{M_0}\rangle K^\times:K^\times|=M_0$.

We now prove the claim. By Kummer theory over $\zeta_{N_Z}$ it suffices to prove that 
$$R_{N_Z} \cap K(\zeta_\infty) \subseteq K(\zeta_{Z^\infty})\,.$$
For a radical $\alpha\in R_{N_Z}\cap K(\zeta_\infty)$ the extension $K(\alpha, \zeta_{N_Z})/K$ is abelian so, by Schinzel's theorem, $\alpha$ is a power of $\zeta_{N_ZE}$ times a Kummer radical $\beta$ (recall that this means $\beta^E\in K^\times$). Thus we are left to check that the Kummer radicals in $R_{N'}\cap K(\zeta_\infty)$ are in $K(\zeta_{Z^\infty})$ for every positive integer $N'$. As the group of Kummer radicals in $R_{N'}$ is contained in a group of radicals $\Gamma'$ (independent of $N'$) such that $|\Gamma' K^\times: K^\times|$ is finite, we may enforce the requested property up to replacing $Z$ by a suitable multiple.
\end{proof}
}

\end{document}